\documentclass{amsart}
\usepackage[centertags]{amsmath}
\usepackage{amsfonts}
\usepackage{amssymb}
\usepackage{latexsym}
\usepackage{amsthm}
\usepackage{newlfont}
\usepackage{graphicx}
\usepackage{listings}
\usepackage{booktabs}
 \usepackage{abstract}

\newcommand{\IR}{\mathbb{R}}
\newcommand{\IC}{\mathbb{C}}
\newcommand{\IZ}{\mathbb{Z}}

\newcommand{\IN}{\mathbb{N}}

\newcommand{\IT}{\mathbb{T}}

\newcommand{\x}{\mathbf {x}}

\newcommand{\bt}{\mathbf {t}}
\newcommand{\bk}{\mathbf {k}}

\newtheorem{Theorem}{Theorem}[section]
\newtheorem{lema}{Lemma}[section]

\newtheorem{prop}{Proposition}[section]

\numberwithin{equation}{section}
\newtheorem*{THM}{Theorem}

\begin{document}

\title{On the Structure of Finitely Generated Shift-Invariant Subspaces }

\author{K. S. Kazarian}

\address{Departamento de Matem\'{a}ticas\\ Universidad Aut\'{o}noma de Madrid\\ 28049 Madrid,
Spain  }
\email{kazaros.kazarian@uam.es}


\thanks{This research did not receive any specific grant from funding agencies in the public, commercial, or not-for-profit sectors}

\maketitle

\begin{abstract}{ A characterization of finitely generated shift-invariant subspaces is given when generators are $g-$minimal. An algorithm is given for the determination of the coefficients in the well known representation of the Fourier transform of an element of the finitely generated shift-invariant subspace as a linear combination of Fourier transformations of generators. An estimate for the norms of those coefficients is derived. For the proof an orthogonalization procedure for generators is used which reminds the well known Gram-Schmidt orthogonalization process. When the generators are compactly supported functions on $\IR$   after the orthogonalization procedure for generators we get a sort of summation basis. }
\end{abstract}

\


\section{ Introduction}

  Further in the paper
$\IT^n = \IR^n/\IZ^{n}$ and  $F\in L^2(\IT^n)$ means that $F$ is defined on
 the whole space $\IR^n$ as $1-$periodic complex-valued function with respect to
 all variables. Such a function is called $\IZ^{n}-$ periodic function.
 If $w\geq 0$ is a measurable function on a measurable set $E \subseteq \IR^{n}$  then we say that $\phi \in L^{2}(E,w)$ if $\phi:E\rightarrow \IC$ is measurable on $E$ and the norm is defined by
\[
 \| \phi \|_{L^{2}(E,w)}:= \left( \int_{E} |\phi(\bt)|^{2} w(\bt) d\bt\right)^{\frac{1}{2}} < +\infty.
 \]
 In the above notation $w$ will be omitted if  $w\equiv 1$.

 With some abuse of the notation it will be interpreted also that $\IT^n$ is the unit cube
 $[0,1)^{n}.$
  The Lebesgue
measure of a set $E\subset\IR^n$ will be denoted as $|E|_n$. The characteristic function of the set $E$ is denoted by $\chi_{E}$. It is also supposed that $\frac{0}{0} = 0$.
A closed subspace $V$ of $L^{2}(\IR^{n})$ is called shift-invariant if for any $f\in V$ and for all $\bk \in \IZ^{n}$ $f(\cdot + \bk) \in V.$
If $\Omega$ is a subset of $L^{2}(\IR^{n})$ then we denote by $S(\Omega)$ the shift-invariant subspace
generated by $\Omega$,
\[
S(\Omega) =  \overline{\mbox{span}} \{\varphi(\cdot + \bk) : \bk \in \IZ^{n}, \varphi\in \Omega \}.
\]
If $\Omega$ consists of just  one element $\varphi$ then the corresponding subspace is called a principle shift invariant space (PSI) and is denoted  by  $S(\varphi)$.
If $\text{card}\, \Omega < \infty$ then $S(\Omega)$ is called finitely generated shift-invariant, or FSI space. A characterization of FSI spaces in
$L^{2}(\IR^{n})$ is given in \cite{BDR:1}, \cite{BDR:2}.
Further on the Fourier transform, an isometry from $L^2(\mathbb{R}^{n})$ onto itself, is defined so that the image
of a function $f\in L^1(\mathbb{R}^{n}) \cap
 L^2(\mathbb{R}^{n})$ is given by
 \[
 \widehat{f}({\mathbf y}) = \int_{\IR^n} f ({\mathbf
 x}) e^{-2\pi i \langle{\mathbf x}, {\mathbf y}\rangle} d{\mathbf x}.
 \]
For $f,g\in L^2(\mathbb{R}^{n})$ using the  definition given in  \cite{BDR:1} we put
 \[
 [{f}, {g}](\bt) = \sum_{{\bk}\in \IZ^{n} } {f}(\bt + \bk) {\overline{g(\bt + \bk)}}.
 \]
 It is easy to observe that $[{f}, {g}] \in L^{1}(\IT^{n})$ and
 \[
 |[{f}, {g}](\bt)|^{2} \leq [{f}, {f}](\bt) [{g}, {g}](\bt)\qquad \text{almost everywhere on $\IT^{n}$}.
 \]
 When it will be  convenient we also use the notation $W_{g} (\bt) = [\widehat{g}, \widehat{g}](\bt)$ which is a well defined $\IZ^{n}-$periodic non-negative function. The orthogonal projection of a function $f\in L^2(\mathbb{R}^{n})$ onto $ S(\Omega)$ is denoted by $P_{\Omega}(f)$.

  The following theorem was proved in \cite{BDR:1}.
\begin{THM}[BDR]\label{T:A}
For any finite subset $G \subset L^2(\mathbb{R}^{n})$ and any $f\in L^2(\mathbb{R}^{n}),$ $f\in S(G)$ if and only if
\[
\widehat{f} = \sum_{\phi \in G} m_{\phi} \widehat{\phi}
\]
for some $\IZ^{n}-$periodic functions $m_{\phi}$.
\end{THM}
 Different variations of the following lemma appeared in various publications (see e.g. \cite{D:92}, \cite{BDR:1}, \cite{BDR:2} and \cite{KSK:1}).
\begin{lema}\label{lem:1}
Let $\varphi \in L^2(\IR^n)$
then a function $f$ is in $S(\varphi),$ if and only if there is
a  measurable function $F$ defined on $\IT^{n}$ such that
$
F \in L^2(\IT^{n}, W_{\varphi}),
$
\begin{equation*}\label{eq:1}
\widehat{{f}}({\bt})=F({\bt}) \widehat{\varphi}({\bt})
\quad \mbox{a.e. on}\quad \IR^n;
\end{equation*}
and
\begin{equation*}\label{eq:2}
\|f\|_{L^2(\IR^n)} = \|F\|_{L^2(\IT^n, W_{\varphi})}.
\end{equation*}
\end{lema}
The above lemma   establishes  an isometry between the
subspace $S(\varphi)$ and the subspace
\begin{equation}\label{sfin:1}
\widehat{S(\varphi)}= \widehat{\varphi} L^2(\IT^n, W_{\varphi})
\subset L^{2}(\IR^{n}).
\end{equation}

 Let $G = \{\phi_{k}\}_{k=1}^{N} \subset L^2(\mathbb{R}^{n})$ be a finite subset of non trivial functions.
 For any $\phi_{j} \in G, 1\leq j \leq N$ we put $ G^{(j)} = G \setminus \{\phi_{j} \}$ and say that $G$ is $g-$minimal if for any $1\leq j \leq N$ it follows that $\phi_{j} \notin S(G^{(j)})$. Further in the paper we suppose that  $G$ is $g-$minimal  and $N\geq 2$. Let
 \begin{equation}\label{eq:orpr}
\widehat{h}_{j} = \widehat{\varphi}_{j} -  \widehat{P_{G^{(j)}}(\varphi_{j})}, \quad \text{where}\quad 1\leq j \leq N.
\end{equation}
   We say  that $G$ is an orthogonal set of generators if
 \begin{equation}\label{orth:00}
[\widehat{\phi}_{l}, \widehat{\phi}_{j}](\bt) = 0  \qquad \mbox{a.e. on}\quad  \IT^{n} \qquad \mbox{ if}\quad  j\neq l \, (1\leq l,j \leq N).
\end{equation}

In the present paper we prove the following theorem which gives a bit different characterization of finitely generated SI subspaces when the set of generators is $g-$minimal.
\begin{Theorem}\label{thm:01}
Let $G = \{\varphi_{k}\}_{k=1}^{N} \subset L^2(\mathbb{R}^{n})$ be a finite $g-$minimal subset. Then for any $f\in S(G)$ it follows that
\begin{equation}\label{eq:ksk}
\widehat{f} = \sum_{j=1}^{N} m_{j}(f) \widehat{\varphi}_{j},\quad \text{where}\quad m_{j}(f)= \frac{[\widehat{f},\widehat{h}_{j}]}{[\widehat{h}_{j},\widehat{h}_{j}]} \in L^{2}(\IT^{n}, W_{{h}_{j}}),
\end{equation}
and
\begin{equation}\label{eq:ksk1}
\sum_{j=1}^{N} \|m_{j}(f)\|^{2}_{L^2(\mathbb{T}^{n}, W_{h_{j}})} \leq \|f\|^{2}_{L^2(\IR^n)},
\end{equation}
where the functions $h_{j}$ are defined by (\ref{eq:orpr}). Moreover, for any $j(1\leq j \leq N)$ and any $m\in L^{2}(\IT^{n}, W_{{h}_{j}})$ there exists $\psi \in  S(G^{(j)}) $ such that
\[
m \widehat{\varphi}_{j} + \widehat{\psi} \in \widehat{S(G)}.
\]
\end{Theorem}

For the proof an orthogonalization procedure for generators is used which reminds the well known Gram-Schmidt orthogonalization process. When the generators are compactly supported functions on $\IR$   after the orthogonalization procedure for generators we get a sort of summation basis.

 \section{Orthogonalization of the generators}

 \

The following proposition is true. 
\begin{prop}\label{pr:03}
Let $\Omega = \{\phi_{k}\}_{k=1}^{\nu} \subset L^2(\mathbb{R}^{n})$ be a finite  set of orthogonal generators.
Then  $f\in S(\Omega)$ if and only if
\begin{equation}\label{eq:00}
\widehat{f} = \sum_{k=1}^{\nu} m_{k}(f) \widehat{\phi}_{k},\quad \text{where}\quad m_{k}(f)= \frac{[\widehat{f},\widehat{\phi}_{k}]}{[\widehat{\phi}_{k},\widehat{\phi}_{k}]} \in L^{2}(\IT^{n}, W_{{\phi}_{k}}).
\end{equation}
  Moreover,
\begin{equation*}\label{eq:001}
\|f\|^{2}_{L^2(\mathbb{R}^{n})} = \sum_{k=1}^{\nu} \|m_{k}(f) \|^{2}_{L^{2}(\IT^{n}, W_{{\phi}_{k}} )}.
\end{equation*}
\end{prop}
\begin{proof}
For $\nu=1$ we have $S(\Omega) = S(\varphi)$. In this case the proposition holds by Lemma \ref{lem:1}. To check (\ref{eq:00}) one should observe that for any trigonometric polynomial $Q$ on $\IT^{n}$ if $\widehat{f} = Q \widehat{\varphi}$ then $[\widehat{f},\widehat{\varphi}] = Q [\widehat{\varphi},\widehat{\varphi}]$. Afterwards for any $g\in L^{2}(\IT^{n}, W_{{\varphi}} )$ there exists a sequence of trigonometric polynomials $\{Q_{j}\}_{j=1}^{\infty}$ such that $\lim_{j\to \infty} Q_{j} =  g$ in $L^{2}(\IT^{n}, W_{{\varphi}} )$. Which means that $ \lim_{j\to \infty}[Q_{j}\widehat{\varphi},\widehat{\varphi}]= [g\widehat{\varphi},\widehat{\varphi}]$, hence, the condition (\ref{eq:00}) holds.

If $\nu >1$ by (\ref{orth:00}) it follows that $S(\phi_{l}) \bot S(\phi_{j})  $ for $l\neq j (1\leq l,j \leq \nu).$ Hence, if
 $f\in S(\Omega)$ and $f= \sum_{j=1}^{\nu} f_{j}$, where $f_{j}\in S(\phi_{j}), (1\leq j \leq \nu)$ then clearly
 $$\|f\|^{2}_{L^2(\mathbb{R}^{n})} = \sum_{j=1}^{\nu} \|f_{j}\|^{2}_{L^2(\mathbb{R}^{n})}$$.

 Let  $f\in S(\Omega)$ and let $P_{k} = \sum_{j=1}^{\nu} f_{j,k}$, where $ f_{j,k} \in S(\phi_{j}), 1\leq j \leq \nu$ for all $k\in \IN$ and
 $\lim_{k\to \infty} P_{k} =  f$ in $L^2(\mathbb{R}^{n})$. It follows that $\{P_{k}\}_{k=1}^{\infty}$ is a Cauchy sequence in $L^2(\mathbb{R}^{n})$. By mutual orthogonality of the  subspaces $S(\phi_{j})$ we will have that for any $(1\leq j \leq \nu)$   $\{f_{j,k}\}_{k=1}^{\infty} $ is a Cauchy sequence in $S(\phi_{j})$. Thus we obtain that $f= \sum_{j=1}^{\nu} f_{j}$, where $f_{j}\in S(\phi_{j}), (1\leq j \leq \nu)$.

 The proof of sufficiency follows by Lemma \ref{lem:1}. If the condition \eqref{eq:00} holds then $ m_{k}(f) \widehat{\phi}_{k} \in \widehat{S(\phi_{k})} $
  for all $k(1\leq k \leq \nu)$. Hence, $f\in S(\Omega)$.
\end{proof}

Following two propositions establish an orthogonalization procedure for generators which reminds the Gram-Schmidt orthogonalization process.
\begin{prop}\label{pr:01}
Let $G = \{\varphi_{k}\}_{k=1}^{N} \subset L^2(\mathbb{R}^{n})$ be a finite $g-$minimal subset.

  Then the functions $\{g_{k}\}_{k=1}^{N}$   defined by the relations  $g_{1} = \varphi_{1}$, and
 \begin{equation}\label{orth:11}
\widehat{g}_{k} = \widehat{\varphi}_{k} - \sum_{j=1}^{k-1} b^{(k)}_{j} \widehat{g}_{j},\qquad 1< k\leq N
\end{equation}
where $ g_{k} \in S(G_{k})$ for all $k(1\leq k \leq N),$   $G_{k} = \{\varphi_{j}\}_{j=1}^{k}$ and
 \begin{equation}\label{orth:10}
b^{(k)}_{j} = [\widehat{\varphi}_{k},\widehat{g}_{j}][\widehat{g}_{j},\widehat{g}_{j}]^{-1} \in L^{2}(\IT^{n}, W_{{g}_{j}}),\quad 1\leq j \leq k -1.
\end{equation}
 Moreover,   for any $1\leq l,j \leq N$
\begin{equation}\label{orth:12}
[\widehat{g}_{l}, \widehat{g}_{j}](t) = 0  \qquad \mbox{a.e. on}\quad  \IT^{n} \qquad \mbox{ if}\quad  j\neq l.
\end{equation}
\end{prop}
\begin{proof}
We prove by mathematical induction.\\
 First we check easily that $[\widehat{g}_{1}, \widehat{g}_{2}] = 0  \quad \mbox{a.e. on}\quad  \IT^{n} $. On the other hand
\[
\int_{\IT^{n}} |b^{(2)}_{1}(\bt)|^{2}  W_{{g}_{1}}(\bt) d\bt \leq \int_{\IT^{n}} [\widehat{\varphi}_{2},\widehat{\varphi}_{2}](\bt)  d\bt = \int_{\IR^{n}} |\widehat{\varphi}_{2}(\bt)|^{2} d\bt < +\infty.
\]
It is clear that $ g_{2} \in S(G_{2})$.

 Afterwards,
suppose that for some $1<m\leq N-1$ the condition (\ref{orth:12}) holds for all $1\leq l,j\leq m-1$, where   the functions $g_{j} (1\leq j \leq m-1)$ are defined by (\ref{orth:11}) and the coefficients satisfy  the relations (\ref{orth:10}). Let $\widehat{g}_{m} = \widehat{\varphi}_{m} - \sum_{j=1}^{m-1} b^{(m)}_{j} \widehat{g}_{j}$, where $b^{(m)}_{j} = [\widehat{\varphi}_{m},\widehat{g}_{j}][\widehat{g}_{j},\widehat{g}_{j}]^{-1}$. Then for any $l(1\leq l \leq m-1)$ it follows that
\[
[\widehat{g}_{m}, \widehat{g}_{l}] = [\widehat{\varphi}_{m},\widehat{g}_{l}]  -  b^{(m)}_{l} [\widehat{g}_{l},\widehat{g}_{l}] = [\widehat{\varphi}_{m},\widehat{g}_{l}] - [\widehat{\varphi}_{m},\widehat{g}_{l}] = 0 \quad \text{a.e.}
\]
That $b^{(m)}_{l} \in  L^{2}(\IT^{n}, W_{{g}_{l}}), 1\leq l \leq m-1$ is proved in the same way as the relation $b^{(2)}_{1} \in  L^{2}(\IT^{n}, W_{{g}_{1}})$.
 It is clear that $ g_{m} \in S(G_{m})$.
\end{proof}
The construction of the system $\{g_{k}\}_{k=1}^{N}$  according to Proposition \ref{pr:01} will be referred as orthogonalization procedure for generators.
It should be mentioned that the formula for the orthogonal projection of a function onto a given PSI was obtained in \cite{BDR:2}.

\begin{prop}\label{pr:02}
Let $G_{m}$ be a finite $g-$minimal generating set for some $1< m \leq N$. Suppose that $\{g_{k}\}_{k=1}^{m}$ are defined by (\ref{orth:11}). Then $S(G_{m}) = \oplus_{k=1}^{m} S(g_{k}).$
\end{prop}
\begin{proof}
By (\ref{orth:11}) and (\ref{sfin:1}) it is easy to check that
 \begin{equation*}\label{inc:1}
S(G_{m}) \subseteq \oplus_{k=1}^{m} S(g_{k}) \qquad \text{and}\qquad S(G_{m}) \supseteq \oplus_{k=1}^{m} S(g_{k}).
\end{equation*}
\end{proof}
By Propositions \ref{pr:01} and \ref{pr:02} we obtain
\begin{lema}\label{L:01}
Let $G = \{\varphi_{k}\}_{k=1}^{N} \subset L^2(\mathbb{R}^{n})$ be a finite $g-$minimal subset. Then for any
$f\in L^{2}(\mathbb{R}^{n})$ the orthogonal projection $P_{G}(f)$  of $f$ onto $S(G)$  is given by
\[
\widehat{P_{G}(f)} =    \sum_{k=1}^{N}  [\widehat{f},\widehat{g}_{k}][\widehat{g}_{k},\widehat{g}_{k}]^{-1} \widehat{g}_{k},
\]
where $\{g_{k}\}_{k=1}^{N}$ are defined by equations (\ref{orth:11}) and (\ref{orth:10}).
\end{lema}

\subsection{Proof of Theorem \ref{thm:01}}

For any $f\in S(G)$ there exists a sequence of functions $Q_{\nu} =  \sum_{j=1}^{N} P^{(\nu)}_{j}$,
 where $P^{(\nu)}_{j}(\x) = \sum_{\bk \in {\cal G}_{\nu,j}} \gamma^{(\nu,j)}_{ \bk} \varphi_{j}(\x+ \bk)  $ and ${\cal G}_{\nu,j} \subset \IZ^{n} (\nu \in \IN, 1 \leq j \leq N)$ are some bounded sets such that $\lim_{\nu \to \infty} Q_{\nu} = f$ in
$L^2(\mathbb{R}^{n})$. Then it follows that  $\lim_{\mu \to \infty, \mu < \nu} \| Q_{\nu} - Q_{\mu}\|_{L^2(\mathbb{R}^{n})} = 0$. By the Plancherel theorem and Proposition \ref{pr:03} it follows that for any $j(1\leq j \leq N)$
\[
\| Q_{\nu} - Q_{\mu}\|_{L^2(\mathbb{R}^{n})} = \| \widehat{Q}_{\nu} - \widehat{Q}_{\mu}\|_{L^2(\mathbb{R}^{n})}
\geq  \| {T^{(\nu)}_{j}- T^{(\mu)}_{j}}\|_{L^2(\mathbb{T}^{n},W_{h_{j}})},
\]
 where $h_{j}$ is defined by \eqref{eq:orpr} and
 \[
  T^{(\nu)}_{j}(\bt) =  \sum_{\bk \in {\cal G}_{\nu,j}} \gamma^{(\nu,j)}_{ \bk} e^{2\pi i \langle \bk, \bt\rangle}.
  \]
 Thus the sequence $\{ T^{(\nu)}_{j} \}_{\nu = 1}^{\infty}$ is a Cauchy sequence in the space $L^2(\mathbb{T}^{n},W_{h_{j}})$. Let $m_{j}(f) = \lim_{\nu \to \infty} T^{(\nu)}_{j}$, where the limit is taken in the space $L^2(\mathbb{T}^{n},W_{h_{j}})$.
We have that
\[
\widehat{Q}_{\nu}(\bt) = \sum_{j=1}^{N} T^{(\nu)}_{j}(\bt) \widehat{\varphi}_{j}(\bt)
\]
and $\lim_{\nu \to \infty} \widehat{Q}_{\nu} = \widehat{f}$ in
$L^2(\mathbb{R}^{n})$. Hence,
(\ref{eq:ksk}) holds.

The inequality \eqref{eq:ksk1} is proved by induction. Suppose that $N=2$. By Propositions \ref{pr:02} and \ref{pr:03} it follows that the set of generators $\{\varphi_{1}, h_{2} \}$ is orthogonal and $S(G) = S(\varphi_{1})\oplus S(h_{2})$. Hence, by Proposition \ref{pr:01} it follows that $$
\|f\|^{2}_{L^2(\IR^n)} = \|m_{1}(f)\|^{2}_{L^2(\mathbb{T}^{n}, W_{\varphi_{1}})} + \|m_{2}(f)\|^{2}_{L^2(\mathbb{T}^{n}, W_{h_{2}})}.$$
 On the other hand  we have that
\[
W_{h_{1}}(\bt) = [\widehat{h}_{1}, \widehat{h}_{1}](\bt) = \bigg[\widehat{\varphi}_{1} - \frac{[\widehat{\varphi}_{1},\widehat{\varphi}_{2}] }{[\widehat{\varphi}_{2},\widehat{\varphi}_{2}]} \widehat{\varphi}_{2}, \widehat{\varphi}_{1} - \frac{[\widehat{\varphi}_{1},\widehat{\varphi}_{2}] }{[\widehat{\varphi}_{2},\widehat{\varphi}_{2}]} \widehat{\varphi}_{2}\bigg] (\bt)
\]
\[
= [\widehat{\varphi}_{1}, \widehat{\varphi}_{1}](\bt) - \frac{|[\widehat{\varphi}_{1},\widehat{\varphi}_{2}](\bt) |^{2}}{[\widehat{\varphi}_{2},\widehat{\varphi}_{2}](\bt)} \leq W_{\varphi_{1}}(\bt).
\]
Thus the inequality \eqref{eq:ksk1} is proved for $N=2$.\\
Suppose that the inequality \eqref{eq:ksk1} is true for any $g-$minimal FSI $S(G)$ with $\text{card}\,G =N$. Let $S(\Omega)$ be a $g-$minimal FSI such that
$ \Omega = \{\phi_{k}\}_{k=1}^{N+1}$ and let
\begin{equation}\label{eq:fs1}
\widehat{\varphi}_{j} = \widehat{\phi}_{j} - \frac{[\widehat{\phi}_{j},\widehat{\phi}_{N+1}] }{[\widehat{\phi}_{N+1},\widehat{\phi}_{N+1}]} \widehat{\phi}_{N+1}, \qquad 1\leq j \leq N.
\end{equation}
It is clear that $S(G)\, \bot\, S(\phi_{N+1})$, where $G = \{\varphi_{j}\}_{j=1}^{N}$ and that \newline  $S(G)\, \oplus\, S(\phi_{N+1}) \subseteq S(\Omega)$.
On the other hand ${\phi}_{j} \in S(G)\, \oplus\, S(\phi_{N+1})$ for any $1\leq j \leq N$. Thus $ S(G)\, \oplus\, S(\phi_{N+1}) = S(\Omega)$ and we have that for any $f\in S(\Omega)$
\[
\widehat{f} = \sum_{j=1}^{N} m_{j}(f) + \frac{[\widehat{f},\widehat{\phi}_{N+1}] }{[\widehat{\phi}_{N+1},\widehat{\phi}_{N+1}]} \widehat{\phi}_{N+1},
\]
where $m_{j}(f)$ are defined   as in \eqref{eq:ksk} and $\widehat{h}_{j}, 1\leq j \leq N$ are defined by \eqref{eq:orpr}.
 It is easy to check that for any $1\leq j \leq N$
 \[
 \widehat{\phi}_{j} -  \widehat{P_{\Omega^{(j)}}(\phi_{j})} = \widehat{\phi}_{j} -  \widehat{P_{\Omega^{(j)}}(\varphi_{j})} - \frac{[\widehat{\phi}_{j},\widehat{\phi}_{N+1}] }{[\widehat{\phi}_{N+1},\widehat{\phi}_{N+1}]} \widehat{\phi}_{N+1} = \widehat{\varphi}_{j} -  \widehat{P_{G^{(j)}}(\varphi_{j})} = \widehat{h}_{j}.
 \]
If we put
\[
\widehat{h}_{N+1} = \widehat{\phi}_{N+1} -  \widehat{P_{\Omega^{(N+1)}}(\phi_{N+1})}
 \]
 then as above we check that $W_{h_{N+1}}(\bt) \leq  W_{\phi_{N+1}}(\bt)$ and finish the proof using the inductive supposition.

The second part of the theorem is an easy consequence of Propositions \ref{pr:03}-\ref{pr:02}. For a fixed $j(1\leq j \leq N)$ we consider a rearrangement $ \{\varphi_{k_{l}}\}_{l=1}^{N} $ of the generators such that $\varphi_{k_{N}} = \varphi_{j} $. Applying the orthogonalization procedure on the set of generators $ \{\varphi_{k_{l}}\}_{l=1}^{N} $ we obtain an orthogonal set of generators $ \{\psi_{k}\}_{k=1}^{N} $ such that $\psi_{N} = h_{j}$. By Proposition \ref{pr:02}
it follows that $S(G) = \oplus_{k=1}^{N} S(\psi_{k}).$ Thus for any $m\in L^{2}(\IT^{n}, W_{{h}_{j}})$ if follows that $m \widehat{h}_{j}  \in \widehat{S(G)}$ and by \eqref{eq:orpr} the proof is finished.

\section{Compactly supported generators}

\

In \cite{KSK:1} it was established that shifts of a compactly supported generator in $\IR$, after deleting a certain number of them, constitute a summation basis with respect to the Abel-Poisson and some Cesaro methods of summation in the corresponding PSI space. A survey about the results on basis properties in weighted mean spaces of some incomplete orthonormal systems can be found in \cite{KZ:1}.

Let $\varphi\in L^{2}(\mathbb{R}^{n})$ be a compactly supported complex-valued
function.  We  suppose that the support of a function $f \in
L^{2}(\mathbb{R}^{n})$ is defined only modulo a null set as  $\{\x\in \mathbb{R}^{n}:
f(\x)\neq 0\}.$ It is well-known that $W_{\varphi}$ is a trigonometric polynomial on $\IT^{n}$. The following statement is a slight generalization of the mentioned fact.

\begin{prop}\label{pro:1} Let $\psi,\varphi\in L^{2}(\mathbb{R}^{n})$ be compactly supported complex-valued
	functions and
	\[
	u(\bt):= [\widehat{\psi},\widehat{\varphi}](\bt),\quad (\bt\in\mathbb{R}^{n}).
	\]
	Then there exists $N_{0}\in \mathbb{N}$ such that the Fourier coefficients $c_{\bk}(u) =0 $  for all $|\bk|>N_{0}$.
\end{prop}
\begin{proof}
	The claim holds because of the following elementary relations: for any $ \bk\in \mathbb{Z}^{n}$,
	\begin{equation*}\label{fc:2}
	c_{\bk}(u) = \int_{\IT^{n}}u(\bt)e^{-2\pi i\langle{\mathbf k}, {\bt}\rangle}d\bt
 =\int_{\mathbb{R}^{n}}\psi(\bt)\varphi(\bt+\bk)d\bt, \quad (\bk\in \mathbb{Z}^{n}).
	\end{equation*}
\end{proof}
A complex-valued function $r$ on $\mathbb{T}^{n}$ is called a \emph{rational trigonometric function} if $r(\bt)=\frac{P(\bt)}{Q(\bt)},$ where $P,Q$ are trigonometric polynomials.

\begin{Theorem}\label{thm:2} Let
	$\Omega:=\{\varphi_j\}_{j=1}^N\subseteq L^2(\mathbb{R}^{n})$  be a finite  $g-$minimal  set  of compactly supported complex-valued
	functions. Let $g_1:=\varphi_1$ and the functions $g_1,\ldots,g_N$ be defined
	by the relations \eqref{orth:11}.
	Then, for any $k,1\leq k\leq N$ $[\widehat{g_k},\widehat{g_k}](\bt)$ is a rational trigonometric function such that
	\begin{equation}\label{trf:2}
	[\widehat{g_k},\widehat{g_k}](\bt)\leq [\widehat{\varphi_k},\widehat{\varphi_k}](\bt),\quad(\bt\in\mathbb{R}^{n}).
	\end{equation}
Moreover,
for any compactly supported complex-valued
	function $\psi \in L^{2}(\mathbb{R}^{n})$,
	$[\widehat{\psi},\widehat{\varphi}_{k}]$ is a rational trigonometric function.
\end{Theorem}
\begin{proof}
	For $k=1$ the statement is obvious. Let $m\in \mathbb{N}$ and $1<m<N$. Suppose that the statement holds for any $k\,\,(1\leq k\leq m)$.
	Hence, by Proposition \ref{pro:1} it follows that for any compactly supported complex-valued
	function $\psi \in L^{2}(\mathbb{R}^{n})$,
	\[
	[\widehat{\psi},\widehat{g}_{m+1}](\bt) = [\widehat{\psi},\widehat{\varphi}_{m+1}](\bt) - \sum_{k=1}^{m} \frac{[\widehat{\varphi}_{m+1},\widehat{g_{k}}](\bt) [\widehat{\psi},\widehat{g_{k}}](\bt)}{[\widehat{g_k},\widehat{g_k}](\bt)}
	\]
	is a rational trigonometric function. On the other hand,
	\[
	[\widehat{g}_{m+1},\widehat{g}_{m+1}](\bt) = [\widehat{\varphi}_{m+1},\widehat{\varphi}_{m+1}](\bt) - \sum_{k=1}^{m} \frac{|[\widehat{\varphi}_{m+1},\widehat{g_{k}}](\bt)|^{2} }{[\widehat{g_k},\widehat{g_k}](\bt)}
	\]
	is a rational trigonometric function. Now, the relation \eqref{trf:2} holds because
	\[
	\sum_{k=1}^{m} \frac{|[\widehat{\varphi}_{m+1},\widehat{g_{k}}](\bt)|^{2} }{[\widehat{g_k},\widehat{g_k}](\bt)} \geq 0.
	\]
	
\end{proof}

When $n = 1$ it is clear  for any $k, 1\leq k \leq N$ there exists a trigonometric polynomial $\omega_{k} \geq 0$ such that 
\[
C^{-1}_{k} \leq \frac{W_{g_k}(t)}{\omega_k(t)} \leq C_{k} \qquad \bt \in \IT
\]
for some $C_{k}> 1$. Which means (see  \cite{KSK:1})  that the system of functions consisting of the shifts of the generator $g_{k}$ have the same metric properties in the subspace $ S(g_{k})$ as the trigonometric system in the weighted norm space $L^{2}(\omega_k)$. Hence, using the results proved in \cite{K:1},  \cite{K:2},  \cite{K:3} we obtain that 
  if $\Omega:=\{\varphi_j\}_{j=1}^N\subseteq L^2(\mathbb{R})$  is a finite  $g-$minimal  set  of compactly supported complex-valued
	functions and $g_1:=\varphi_1$ and the functions $g_1,\ldots,g_N$ are defined
	by the relations \eqref{orth:11} then for any $k, 1\leq k \leq N$ deleting a certain number of elements from the system of functions consisting of the shifts of the generator $g_{k}$ we get an Abel-Poisson summation basis and $(C,\alpha)$ summation basis in $ S(g_{k})$ for sufficiently big $\alpha$.

 When $n\geq 2$ we do not know any description of zero-sets of  trigonometric
polynomials of $n$ variables. It should be one of the preliminary steps to extend results obtained in \cite{KSK:1}.

\end{document}